\documentclass[a4paper,leqno,11pt,twoside]{amsart}
\usepackage{amsmath,amsfonts,amssymb,amsthm,amscd}
\usepackage[utf8]{inputenc}
\usepackage[english]{babel}
\usepackage[colorlinks=true,citecolor=blue, urlcolor=blue, linkcolor=blue,pagebackref]{hyperref}
\usepackage[top=1.1in,bottom=1.1in,left=1.in,right=1.in]{geometry} 
\usepackage{graphicx}
\usepackage{wrapfig}
\usepackage{paralist}
\usepackage{tabto}
\usepackage{standalone}
\usepackage{xfrac}
\usepackage{float}

\usepackage{pgf,tikz}
\usetikzlibrary{arrows,chains,
 decorations.pathreplacing,
shapes,%
}

\usepackage[all]{xy} 

\usepackage{enumerate}
\usepackage{xcolor}
\usepackage{aliascnt}

\theoremstyle{plain}
\newtheorem{thm}{Theorem}[section]

\newtheorem{thmIntr}{Theorem}

\newaliascnt{corIntr}{thmIntr}
\newtheorem{corIntr}[corIntr]{Corollary}

\newaliascnt{lem}{thm}
\newtheorem{lem}[lem]{Lemma}
\aliascntresetthe{lem}

\newaliascnt{cor}{thm}
\newtheorem{cor}[cor]{Corollary}
\aliascntresetthe{cor}

\newaliascnt{prop}{thm}
\newtheorem{prop}[prop]{Proposition}
\aliascntresetthe{prop}

\theoremstyle{definition}

\newaliascnt{rem}{thm}
\newtheorem{rem}[rem]{Remark}
\aliascntresetthe{rem}

\newaliascnt{defn}{thm}

\aliascntresetthe{defn}

\newaliascnt{ex}{thm}

\aliascntresetthe{ex}

\numberwithin{equation}{section}

\def\bK{\ensuremath{\mathbb{K}}}

\def\bP{\ensuremath{\mathbb{P}}}
\def\bQ{\ensuremath{\mathbb{Q}}}
\def\bR{\ensuremath{\mathbb{R}}}
\def\bZ{\ensuremath{\mathbb{Z}}}
\def\bC{\ensuremath{\mathbb{C}}}

\def\cC{\ensuremath{\mathcal{C}}}

\def\cG{\ensuremath{\mathcal{G}}}
\def\cH{\ensuremath{\mathcal{H}}}
\def\cI{\ensuremath{\mathcal{I}}}

\def\cP{\ensuremath{\mathcal{P}}}

\def\Db{\mathop{\mathrm{D}^{\mathrm{b}}}\nolimits}

\DeclareMathOperator{\Pic}{Pic}

\DeclareMathOperator{\Hom}{Hom}

\DeclareMathOperator{\NS}{NS}

\DeclareMathOperator{\ch}{ch}
\DeclareMathOperator{\Coh}{Coh}

\DeclareMathOperator{\odisc}{\overline{\Delta}}

\DeclareMathOperator{\Char}{char}

\definecolor{applegreen}{rgb}{0.55, 0.71, 0.0}

\newcommand{\set}[1]{\left\{#1\right\}}

\title[The basepoint-freeness threshold of a very general abelian surface]{The basepoint-freeness threshold of a very general abelian surface}
\author[A.~Rojas]{Andr\'es Rojas}

\address{Departament de Matem\`atiques i Inform\`atica, Universitat de Barcelona, Gran Via de les Corts Catalanes 585, 08007 Barcelona, Spain}
\curraddr{Mathematisches Institut, Universität Bonn, Endenicher
Allee 60, 53115 Bonn, Germany}
\email{arojas@math.uni-bonn.de} 

\begin{document}

\begin{abstract}
For abelian surfaces of Picard rank 1, we perform explicit computations of the cohomological rank functions of the ideal sheaf of one point, and in particular of the \emph{basepoint-freeness threshold}. Our main tool is the relation between cohomological rank functions and Bridgeland stability.
In virtue of recent results of Caucci and Ito, these computations provide new information on the syzygies of polarized abelian surfaces.
\end{abstract}


\thanks{The author was partially supported by the Spanish MINECO grants MDM-2014-0445, RYC-2015-19175, and PID2019-104047GB-I00.}

\maketitle

\setcounter{tocdepth}{1}

\section{Introduction}

Throughout this note we work over an algebraically closed field $\bK$.

Motivated by the \emph{continuous rank functions} of Barja, Pardini and Stoppino (\cite{BPS:LinearSystems}), in their paper \cite{JP} Jiang and Pareschi introduced the \emph{cohomological rank functions} $h^i_{F,l}$ associated to a coherent sheaf (or more generally, a bounded complex of coherent sheaves) $F$ on a polarized abelian variety $(A,l)$.
For $x\in\bQ$, $h^i_{F,l}(x)$ makes sense of the $i$-th (hyper)cohomological rank of $F$ twisted with a (general) representative of the fractional polarization $xl$.

One of the main applications of these functions corresponds to the study of syzygies of abelian varieties. Jiang and Pareschi already observed in \cite[Section 8]{JP} that the functions of the ideal sheaf $\cI_q$ of a point $q\in A$, and more concretely the \emph{basepoint-freeness threshold}
\[
\epsilon_1(l)=\inf\set{x\in\bQ\mid h^1_{\cI_q,l}(x)=0},
\]
encodes interesting positivity properties of the polarization $l$:

\begin{enumerate}[{\rm (1)}]
\item\label{enum:intro1} $\epsilon_1(l)\leq1$, with equality if and only if any line bundle representing $l$ has base points.

\item\label{enum:intro2} \cite[Corollary~E]{JP} If $\epsilon_1(l)<\frac{1}{2}$, then any line bundle representing $l$ is projectively normal.
\end{enumerate}

Shortly after, Caucci generalized \eqref{enum:intro2} to higher syzygies, proving that every line bundle representing $l$ satisfies the property $(N_p)$ as long as $\epsilon_1(l)<\frac{1}{p+2}$ (\cite[Theorem 1.1]{Caucci}).
The reader is referred to \cite[Chapter 1.8.D]{laz} for a definition of the property $(N_p)$. As a consequence, he obtained a proof of Lazarsfeld's conjecture (originally proved in $\Char(\bK)=0$ by Pareschi \cite{Pareschi}) in arbitrary characteristic: if $L$ is an ample line bundle on an abelian variety, then $L^m$ satisfies $(N_p)$ for every $m\geq p+3$.

Caucci's result has received considerable attention as an effective tool to understand the syzygies of abelian varieties endowed with a primitive polarization (i.e. a polarization which is not a multiple of another one), by means of upper bounds for the basepoint-freeness threshold (see \cite{Jiang:Syzygies,Ito1,Ito2}). Furthermore, for $p\geq1$ the hypothesis $\epsilon_1(l)<\frac{1}{p+2}$ ensuring $(N_p)$ has recently been slightly weakened by Ito (\cite[Theorem 1.5]{Ito3}).

In the present note we give explicit expressions for the function $h^0_{\cI_q,l}$, which is enough for determining $h^1_{\cI_q,l}$ and hence $\epsilon_1(l)$. We do this for a certain class of polarized abelian surfaces which includes those with Picard rank 1. More precisely, our main result is:

\begin{thmIntr}\label{main}
Let $(S,l)$ be a $(1,d)$-polarized abelian surface such that $D\cdot l$ is a multiple of $l^2$ for every divisor class $D$, and let $q\in S$ be a (closed) point.
\begin{enumerate}[{\rm (1)}]
\item\label{main1} If $d$ is a perfect square, then the cohomological rank function $h^0_{\cI_q,l}$ reads
\begin{equation}\label{trivialfunction}
    h^0_{\cI_q,l}(x)=\left\{
    \begin{array}{c l}
     0 & x \leq \frac{\sqrt{d}}{d}\\
     dx^2-1 & x \geq \frac{\sqrt{d}}{d}\\
    \end{array}
    \right.
\end{equation}
In particular, $\epsilon_1(l)=\frac{\sqrt{d}}{d}$.

\vspace{1mm}

\item\label{main2} If $d$ is not a perfect square, then the cohomological rank function $h^0_{\cI_q,l}$ is either \eqref{trivialfunction} or
\[
    h^0_{\cI_q,l}(x)=\left\{
    \begin{array}{c l}
    0 & x \leq \frac{2\tilde{y}}{\tilde{x}+1}\\
    \frac{d(\tilde{x}+1)}{2}x^2-2d\tilde{y}\cdot x+\frac{\tilde{x}-1}{2} & \frac{2\tilde{y}}{\tilde{x}+1} \leq x \leq \frac{2\tilde{y}}{\tilde{x}-1}\\
    dx^2-1 & x \geq \frac{2\tilde{y}}{\tilde{x}-1}\\
    \end{array}
    \right.
\]
where $(\tilde{x},\tilde{y})$ is a nontrivial positive solution to Pell's equation $x^2-4d\cdot y^2=1$. In particular, if $(x_0,y_0)$ is the minimal positive solution to this equation, then $\epsilon_1(l)\leq\frac{2y_0}{x_0-1}$.

\vspace{1mm}

\item\label{main3} Under the hypothesis of \eqref{main2}, assume also that $\Char(\bK)$ divides neither $x_0^2$ nor $x_0^2-1$. Then the expression for $h^0_{\cI_q,l}$ is the one corresponding to either the minimal solution $(x_0,y_0)$ or to the second smallest positive solution $(x_1,y_1)$. In particular, $\epsilon_1(l)\in\{\frac{2y_0}{x_0-1},\frac{2y_1}{x_1-1}\}$.

\end{enumerate}
\end{thmIntr}

Parts \eqref{main1} and \eqref{main2} of this result are proved in \autoref{sec:Upperbound}. Our main tool is a natural description of cohomological rank functions on abelian surfaces in terms of certain stability conditions on the derived category, which has recently been proved by Lahoz and the author in \cite{LR}. Essentially, this description establishes that $h^0_{\cI_q,l}$ is determined by the Harder-Narasimhan filtrations of $\cI_q$ along the so-called \emph{$(\alpha,\beta)$-plane} of stability conditions.

The key point of this approach is that the potential destabilizing walls for $\cI_q$ are in correspondence with positive solutions to Pell's equation $x^2-4d\cdot y^2=1$ (see \autoref{walls}).  The absence of such solutions when $d$ is a perfect square shows \eqref{main1}, whereas for $d$ not a perfect square one obtains \eqref{main2}.

The corresponding upper bounds for the basepoint-freeness threshold refine those given by Ito for general complex abelian surfaces (\cite{Ito2}). 
In addition, the expressions of \eqref{main1} and \eqref{main2} reveal the differentiability of $h^0_{\cI_q,l}$ at certain rational points; this is relevant with regard to syzygies, since it enables us to apply Ito's refined version of Caucci's criterion. As a result, we have:

\begin{corIntr}\label{syzygies}
Let $(S,l)$ be a $(1,d)$-polarized abelian surface which satisfies the hypothesis of \autoref{main}, and let $L$ be any line bundle representing the polarization $l$.
\begin{enumerate}[\rm (1)]
    \item\label{enum:projnormal} If $d\geq7$, then $L$ is projective normal.
    \item\label{enum:np} If $d>(p+2)^2$ for $p\geq1$, then $L$ satisfies the property $(N_p)$.
\end{enumerate}
\end{corIntr}

For $\bK=\bC$, we point out that \autoref{syzygies}.\eqref{enum:projnormal} recovers a well known result of Iyer (\cite{iyer}, see also \cite{laz2} for some cases previously covered), and the case $p=1$ of \autoref{syzygies}.\eqref{enum:np} recovers a result of Gross and Popescu (\cite{gp}). For arbitrary $p$, \autoref{syzygies}.\eqref{enum:np} improves the bound ensuring the property $(N_p)$ that was given recently by Ito in \cite[Corollary 4.4]{Ito2}.

In \autoref{sec:Lowerbound} we deal with the proof of \autoref{main}.\eqref{main3}. This is the problem of determining, when $d$ is not a perfect square, which of the potential functions described in \autoref{main}.\eqref{main2} really occur. Modulo certain arithmetic restrictions on $\Char(\bK)$, we prove that only two possibilities may happen: those corresponding to the two smallest positive solutions of Pell's equation.

This is guaranteed by the explicit construction of curves containing all the torsion points of an unexpectedly high order (see \autoref{symcurve}). For this construction, we use the classical theory of theta groups developed by Mumford in \cite{mumford}.

It is worth noting that, for all the non-perfect squares $d$ for which we know the exact value of $\epsilon_1(l)$, the equality $\epsilon_1(l)=\frac{2y_0}{x_0-1}$ holds. In general, this would follow from a small refinement of \autoref{symcurve}, that at present we do not know how to prove (see \autoref{finalrem}.\eqref{minimalsol} for details).

\textbf{Acknowledgements.} This work has benefited from helpful conversations with my advisors Martí Lahoz and Joan Carles Naranjo. Thanks are also due to Federico Caucci for useful comments.

\section{Preliminaries}

\subsection{Cohomological rank functions} Let $(A,l)$ be a $g$-dimensional polarized abelian variety, i.e. $l\in\NS(A)=\Pic(A)/\Pic^0(A)$ is the class of an ample line bundle $L$. We will denote by 
\[
\varphi_l:A\to\Pic^0(A),\;\;p\mapsto t_p^*L\otimes L^{-1}
\]
its polarization isogeny, where $t_p$ stands for the translation by $p\in A$. 

Let $\Db(A)$ be the bounded derived category of $A$. In the paper \cite{JP} (see \cite[Section 2]{Caucci} for positive characteristic), a \emph{cohomological rank function}
\[
h^i_{F,l}:\bQ\to\bQ_{\geq0}
\]
is associated to every object $F\in\Db(A)$ and every $i\in\bZ$. If $x_0=\frac{a}{b}\in\bQ$ with $b\in\bZ_{>0}$, then $h^i_{F,l}(x_0)$ is defined as
\[
h^i_{F,l}(x_0):=\frac{1}{b^{2g}}h^i(A,\mu_b^*F\otimes L^{ab}\otimes\alpha)
\]
for general $\alpha\in\Pic^0(A)$, where $\mu_b:A\to A$ is the multiplication-by-$b$ isogeny. Since $\mu_b^*l=b^2l$ and $\deg(\mu_b)=b^{2g}$, the number $h^i_{F,l}(x_0)$ gives a meaning to the $i$-th (hyper)cohomological rank of $F$ twisted with a (general) representative of the fractional polarization $x_0l$.

These functions are polynomial in the neighborhood of any fixed $x_0\in\bQ$. More explicitly, for any sheaf $E$ let $\chi_{E,l}$ be the Hilbert polynomial of $E$ with respect to $l$. Then for every (rational) $x$ in a right neighborhood of $x_0$, the following equality holds (\cite[Section 2]{JP}):
\begin{equation}\label{explicitfunction}
h^i_{F,l}(x)=\frac{(x-x_0)^g}{\chi(l)}\cdot\chi_{\varphi_l^*R^{g-i}\Phi_{\cP^{\vee}}((\mu_{b}^*F\otimes L^{ab})^{\vee}),l}\left(\frac{1}{b^2(x-x_0)}\right),
\end{equation}
where $\Phi_{\cP^{\vee}}$ denotes the Fourier-Mukai transform with kernel the dual $\cP^\vee$ of the Poincaré bundle.

In this note, we concentrate on the functions $h^i_{\cI_q,l}$ for a (closed) point $q\in A$; by independence of $q$, we fix $q$ to be the origin $0\in A$.
As explained in the introduction, previous work of Jiang-Pareschi, Caucci and Ito shows that they encode information about the polarization $l$:

\begin{thm}[\cite{JP,Caucci,Ito3}] \label{ito}
Let $(A,l)$ be a polarized abelian variety, and let $L$ be any ample line bundle representing the polarization $l$.
\begin{enumerate}[\rm (1)]
    \item $\cI_0\langle l\rangle$ is IT(0) if and only if $L$ is basepoint-free.
    \item\label{ref:projnormal} If $\cI_0\langle\frac{1}{2}l\rangle$ is IT(0), then $L$ is projectively normal.
    \item\label{ref:Np} If $\cI_0\langle\frac{1}{p+2}l\rangle$ is M-regular for some $p\geq1$, then $L$ satisfies the property $(N_p)$.
\end{enumerate}
\end{thm}

The reader is referred to \cite[Section 5]{JP} for the definitions of a $\bQ$-twisted coherent sheaf $F\langle x_0l\rangle$ being IT(0), M-regular or a GV-sheaf. In the particular case $F=\cI_0$ we will use the following characterization, which is an immediate consequence of \cite[Proposition 5.3]{JP}:

\begin{lem}\label{regularity}
Let $x_0\in\bQ$ be a positive rational number. 
\begin{enumerate}[\rm (1)]
    \item\label{GV} $\cI_0\langle x_0l\rangle$ is a GV-sheaf if and only if $h^1_{\cI_0,l}(x_0)=0$.
    \item\label{Mreg} $\cI_0\langle x_0l\rangle$ is M-regular if and only if $h^1_{\cI_0,l}(x_0)=0$ and $h^1_{\cI_0,l}$ is of class $\cC^1$ at $x_0$.
    \item\label{IT0} $\cI_0\langle x_0l\rangle$ is IT(0) if and only if there is $\epsilon>0$ such that $h^1_{\cI_0,l}(x)=0$ for all $x\in(x_0-\epsilon,x_0)$.
\end{enumerate}
\end{lem}

\subsection{The $(\alpha,\beta)$-plane of a polarized abelian surface} In this subsection, $(S,l)$ will be a polarized abelian surface. We briefly recall the relation between cohomological rank functions and stability in the \emph{$(\alpha,\beta)$-plane} associated to $l$, which appeared recently in \cite{LR}.

For every $(\alpha,\beta)\in\bR_{>0}\times\bR$, there exists a \emph{Bridgeland stability condition} $\sigma_{\alpha,\beta}=(\Coh^\beta(S),Z_{\alpha,\beta})$ (see \cite[Section 6]{MS}, or \cite{Bridgeland:K3} for the original treatment), where:

\begin{itemize}[\textbullet]
    \item $\Coh^\beta(S)$ is the heart of a bounded t-structure on $\Db(S)$. Concretely, if $\mu_l=\frac{l\cdot\ch_1}{l^2\cdot\ch_0}$ is the slope of a coherent sheaf, complexes $F\in\Coh^\beta(S)$ are those satisfying: $\mu_l(E)\leq\beta$ for every subsheaf $E\subset\cH^{-1}(F)$, $\mu_l(Q)>\beta$ for every quotient $\cH^0(F)\twoheadrightarrow Q$, and $\cH^i(F)=0$ for $i\neq0,-1$.
    
    \vspace{2mm}
    
    \item Let $K_0(\Db(S))$ denote the Grothendieck group of $\Db(S)$. Then $Z_{\alpha,\beta}:K_0(\Db(S))\to\bC$ is a group homomorphism with the following properties:

\begin{enumerate}
    \item\label{BSC1} $Z_{\alpha,\beta}$ factors through the homomorphism $v:K_0(\Db(S))\to\Lambda=\bZ^3$ defined by
    \[
    v(E)=\left(l^2\cdot\ch_0(E),l\cdot\ch_1(E),\ch_2(E)\right)\]
    
    \item\label{BSC2} For every nonzero $E\in\Coh^\beta(S)$ the inequality $\Im Z_{\alpha,\beta}(E)\geq0$ holds, and $\Re Z_{\alpha,\beta}(E)<0$ whenever $\Im Z_{\alpha,\beta}(E)=0$.
    
    \item\label{BSC3} Every object of $\Coh^\beta(S)$ admits a \emph{Harder-Narasimhan} (\emph{HN} for short) \emph{filtration} with respect to the \emph{tilt slope}
    \[
    \nu_{\alpha,\beta}(E):=\left\{
    \begin{array}{c l}
     \frac{-\Re Z_{\alpha,\beta}(E)}{\Im Z_{\alpha,\beta}(E)} & \Im Z_{\alpha,\beta}(E) >0\\
     +\infty & \Im Z_{\alpha,\beta}(E)=0\\
    \end{array}
    \right.
    \]
\end{enumerate}
\end{itemize}

\vspace{2mm}

It is specially relevant that $\sigma_{\alpha,\beta}$ satisfies the \emph{support property} (see \cite[Section 5.2]{MS}) with respect to the following quadratic form in $\Lambda\otimes\bR$:
\[
\odisc=(l\cdot\ch_1)^2-2(l^2\cdot\ch_0)\ch_2=v_1^2-2v_0v_2
\]
This fact, combined with the so-called \emph{Bertram's nested wall theorem}, gives an effective control of wall-crossing in the $(\alpha,\beta)$-plane (see e.g. \cite[Theorem 2.8]{LR}); essentially, the \emph{walls} where $\sigma_{\alpha,\beta}$-semistability varies for objects of a fixed class $v\in\Lambda$ are nested semicircles.

For $\alpha=0$ and $\beta\in\bQ$, the pair $\sigma_{0,\beta}=(\Coh^\beta(S),Z_{0,\beta})$ defines a \emph{weak stability condition}; $Z_{0,\beta}$ satisfies the properties \eqref{BSC1} and \eqref{BSC3} listed above, with the difference that the equality $Z_{0,\beta}(E)=0$ holds for certain nonzero objects $E\in\Coh^\beta(S)$.

As proved in \cite{LR}, the HN filtrations of any object $F\in\Db(S)$ with respect to these weak stability conditions describe the cohomological rank functions $h^i_{F,l}$. In the simplest situation of an object lying in the heart, this description reads as follows:

\begin{prop}[\cite{LR}]\label{CRFstab}
Let $\beta\in\bQ$, and let $0=F_0\hookrightarrow F_1\hookrightarrow...\hookrightarrow F_r=F$ be the HN filtration with respect to $\sigma_{0,\beta}$ of an object $F\in\Coh^\beta(S)$. Then:
\begin{enumerate}[\rm (1)]
    \item $h^i_{F,l}(-\beta)=0$ for every $i\neq0,1$.
    \item $h^0_{F,L}(-\beta)=\displaystyle\sum_{\nu_{0,\beta}({F_j/F_{j-1}})\geq0}\chi_{F_j/F_{j-1},l}(-\beta)$, and $h^1_{F,L}(-\beta)=\displaystyle\sum_{\nu_{0,\beta}({F_j/F_{j-1}})<0}-\chi_{F_j/F_{j-1},l}(-\beta)$.
\end{enumerate}
\end{prop}

\subsection{The theta group of an ample line bundle} Let $(A,l)$ be a polarized abelian variety, and let $L$ be an ample line bundle representing $l$. We give a quick review of the representation of the theta group $\cG(L)$ on $H^0(A,L)$, explicitly described by Mumford in \cite{mumford}.

Assume that $\Char(\bK)$ does not divide $h^0(L)$. This guarantees that the polarization isogeny
$\varphi_l:A\to\Pic^0(A)$
is separable.
We will write $K(L):=\ker(\varphi_l)$; for instance, if $L$ is very ample embedding $A$ in $\bP(H^0(A,L)^\vee)$, then the points $p\in K(L)$ are those for which the translation $t_p$ on $A$ extends to a projectivity of $\bP(H^0(A,L)^\vee)$.

This projective representation comes from the aforementioned representation of the \emph{theta group}
\[
\cG(L):=\{(x,\varphi)\mid x\in K(L),\;\;\varphi:L\overset{\cong}{\longrightarrow}t_x^*L\},\;\;\;(y,\psi)\cdot(x,\varphi)=(x+y,t_x^*\psi\circ\varphi)
\]
on $H^0(A,L)$. Note that $\cG(L)$ fits into a short exact sequence
\[
1\to\bK^*\to\cG(L)\to K(L)\to 0,
\]
but it is far from being abelian. Indeed, the skew-symmetric pairing $e^L:K(L)\times K(L)\to\bK^*$ measuring the noncommutativity of $\cG(L)$ is  non-degenerate (see \cite[Page 293]{mumford}).

The representation of $\cG(L)$ on $H^0(A,L)$ is defined as follows: every $(x,\varphi)\in\cG(L)$ induces
\[
U_{(x,\varphi)}:H^0(A,L)\to H^0(A,L),\;\;\;s\mapsto t_{-x}^*(\varphi(s))
\]

\begin{thm}[\cite{mumford}]\label{mumford}
With the notations above, the following statements hold:
\begin{enumerate}[\rm (1)]
    \item $K(L)=A(L)\oplus B(L)$, where $A(L),B(L)\subset K(L)$ are maximal totally isotropic subgroups with respect to $e^L$. Moreover, if $L$ is of type $\delta=(d_1,...,d_g)$, then $A(L)\cong \sfrac{\bZ}{d_1}\oplus...\oplus\sfrac{\bZ}{d_g}$ and $B(L)\cong\widehat{A(L)}=\Hom_\bZ(A(L),\bK^*)$ via the pairing $e^L$.
    \item As a group, $\cG(L)$ is isomorphic to $\cG(\delta):=\bK^*\times A(L)\times\widehat{A(L)}$ with the operation
    \[
    (\alpha,t,l)\cdot(\alpha',t',l')=(\alpha\alpha'\cdot l'(t),t+t',l\cdot l')
    \]
    \item\label{repthetagroup} The representation of $\cG(L)$ on $H^0(A,L)$ is isomorphic to the representation of $\cG(\delta)$ on
    \[
    V(\delta)=\{\bK\text{-valued functions on $A(L)=\sfrac{\bZ}{d_1}\oplus...\oplus\sfrac{\bZ}{d_g}$}\}
    \]
    given, for $(\alpha,t,l)\in\cG(\delta)$ and $f\in V(\delta)$, as follows:
    \[
    \left((\alpha,t,l)\cdot f\right)(x)=\alpha\cdot l(x)\cdot f(t+x)
    \]
    \item Assume that $\Char(\bK)\neq 2$ and $L$ is totally symmetric: namely, there exists an isomorphism $L\cong i^*L$, acting as $+1$ simultaneously on all the fibers $L(p)$ of 2-torsion points $p\in A_2$. Then the inversion map $i:A\to A$ extends to a projectivity of $\bP(H^0(A,L)^\vee)$; under the isomorphism $H^0(A,L)\cong V(\delta)$ of \eqref{repthetagroup}, this projectivity is obtained from
    \[
    \tilde{i}:V(\delta)\to V(\delta),\;\;\left(\tilde{i}\cdot f\right)(x)=f(-x)
    \]
\end{enumerate}
\end{thm}

The main advantage of this description is the existence of a canonical basis for $V(\delta)$, which allows an explicit treatment of the endomorphisms $U_{(x,\varphi)}$ in coordinates.
We will use this approach in the proof of \autoref{symcurve}.

\section{Upper bounds for \texorpdfstring{$\epsilon_1(l)$}{epsilon1(l)}}\label{sec:Upperbound}

Throughout this section, $(S,l)$ will be a polarized abelian surface satisfying the hypothesis of \autoref{main}, namely: $l$ is of type $(1,d)$, and for every divisor class $D$ we have $l^2|D\cdot l$.

Since $\cI_0$ is a slope-semistable sheaf, it follows from the very definition that $\cI_0\in\Coh^\beta(S)$ for every $\beta<0$. Hence we may apply \autoref{CRFstab} to describe the cohomological rank functions of $\cI_0$ for $x\geq0$.

Moreover, since $\cI_0$ is a Gieseker semistable sheaf, $\cI_0$ is $\sigma_{\alpha,\beta}$-semistable for every $\beta<0$ and $\alpha\gg0$ (\cite[Proposition 14.2]{Bridgeland:K3}). Thus our problem is reduced to understand how the HN filtration of $\cI_0$ with respect to $\sigma_{\alpha,\beta}$ varies, as $\alpha$ decreases.

To this end, observe that $\odisc(\cI_0)=2l^2(=4d$, by Riemann-Roch) takes the minimum possible positive value; indeed, by our assumptions on $(S,l)$ we have $4d|\odisc(v(E))$ for every $E\in\Db(S)$. In terms of wall-crossing this is a strong constraint, which guarantees one of the following conditions (see \cite[Subsection 7.2]{LR}):

\begin{enumerate}[\rm (1)]
    \item\label{enum:function1} Either $\cI_0$ is $\sigma_{\alpha,\beta}$-semistable for every $\beta<0$ and $\alpha>0$, in which case $h^0_{\cI_0,l}$ reads
    \[
    h^0_{\cI_0,l}(x)=\left\{
    \begin{array}{c l}
     0 & x \leq \frac{\sqrt{d}}{d}\\
     \chi_{\cI_0,l}(x)=dx^2-1 & x \geq \frac{\sqrt{d}}{d}\\
    \end{array}
    \right.
    \]
    In particular, the functions $h^i_{\cI_0,l}$ ($i=0,1$) are not of class $\cC^1$ at $\frac{\sqrt{d}}{d}$.
    
    \vspace{1.5mm}
    
    \item\label{enum:function2} Or $\cI_0$ destabilizes along a semicircular wall $W$ defined by a short exact sequence $0\to E\to \cI_0\to Q\to 0$ in $\Coh^{-\frac{\sqrt{d}}{d}}(S)$, with $\odisc(E)=0=\odisc(Q)$. If $p_Q<p_E$ are the intersection points of this semicircle with the line $\alpha=0$, then
    \[
    h^0_{\cI_0,l}(x)=\left\{
    \begin{array}{c l}
    0 & x\leq -p_E\\
    \chi_{E,l}(x) & -p_E\leq x\leq-p_Q\\
    \chi_{\cI_0,l}(x)=dx^2-1 & x\geq-p_Q\\
    \end{array}
    \right.
    \]
    \noindent In particular, the cohomological rank functions $h^i_{\cI_0,l}$ ($i=0,1$) are $\cC^1$ at $-p_E$ and $-p_Q$.
\end{enumerate}

\vspace{1.5mm}

\begin{lem}\label{walls}
Let $0\to E\to \cI_0\to Q\to 0$ be a destabilizing short exact sequence as in \eqref{enum:function2}. Then $v(E)=(d(\widetilde{x}+1),-2d\widetilde{y},\frac{\widetilde{x}-1}{2})$ and $v(Q)=((1-\widetilde{x})d,2d\widetilde{y},-\frac{\widetilde{x}+1}{2})$, where $(\widetilde{x},\widetilde{y})$ is a positive nontrivial solution to Pell's equation $x^2-4d\cdot y^2=1$.
\end{lem}
\begin{proof}
By the assumption $l^2|D\cdot l$ for every divisor class $D$, we may write  $v(E)=(2dr,2dc,\chi)$ and $v(Q)=(2d(1-r),-2dc,-1-\chi)$ for certain integers $r,c$ and $\chi$. 
The condition $\odisc(E)=\odisc(Q)$ is easily checked to read as $r=\chi+1$. 

Imposing now $\odisc(E)=0$ gives
$\chi(\chi+1)-dc^2=0$, which after multiplying by 4 and adding 1 at both sides, becomes $(2\chi+1)^2-4d\cdot c^2=1$.
Therefore, $(2\chi+1,c)$ is a solution to the equation $x^2-4d\cdot y^2=1$.
Note that this solution must be non-trivial: otherwise, either $E$ or $Q$ would have class $v=(0,0,-1)$, which is impossible.

Finally, we have to determine the signs of the solution $(2\chi+1,c)$ to Pell's equation. Since $E$ is a subobject of the torsion-free sheaf $\mathcal{I}_0$ in the category $\Coh^{-\frac{\sqrt{d}}{d}}(S)$, it follows that $E$ is a sheaf with $r=\ch_0(E)>0$ (hence $2\chi+1>0$). Moreover, since $\odisc(E)=0$, the right intersection point $p_E$ of $W$ with the $\beta$-axis equals $\mu_l(E)=\frac{c}{r}$ (see e.g. \cite[Theorem 2.8]{LR}). $W$ being a wall for $\cI_0$, it lies entirely in the region with $\beta<0$; this gives $c<0$, which finishes the proof.
\end{proof}

Combining this characterization of the walls with the previous description of the function $h^0_{\cI_0,l}$, one concludes the proof of \autoref{main}:

\begin{proof}[Proof of \autoref{main}.\eqref{main1}-\eqref{main2}]
If $d$ is a perfect square (equivalently, $4d$ is a perfect square), then Pell's equation involved in \autoref{walls} admits only trivial solutions, so $\mathcal{I}_0$ is $\sigma_{\alpha,\beta}$-semistable along the whole region $\beta<0$. In this case, $h^0_{\cI_0,l}$ admits the expression given in \eqref{enum:function1}.

Now assume that $d$ is not a perfect square. If $\cI_0$ destabilizes (equivalently, $h^0_{\cI_0,l}$ is not the function given in \eqref{enum:function1}), then by \autoref{walls} the destabilizing wall corresponds to a positive nontrivial solution $(\widetilde{x},\widetilde{y})$ of $x^2-4d\cdot y^2=1$, for which the classes $v(E)$ and $v(Q)$ are known. Combining these classes with \eqref{enum:function2}, we obtain the explicit expression of $h^0_{\cI_0,l}$. 

Finally, observe that in the same way as the quotients $\frac{\widetilde{y}}{\widetilde{x}}$ converge to $\frac{\sqrt{d}}{2d}$, the walls accumulate towards the point $-\frac{\sqrt{d}}{d}$ in the $\beta$-axis:

\begin{figure}[ht]
\definecolor{ffqqqq}{rgb}{1,0,0}
\definecolor{qqwuqq}{rgb}{0,0.39215686274509803,0}
\begin{tikzpicture}[line cap=round,line join=round,>=triangle 45,x=17cm,y=17cm]

\clip(-0.683667693992299,-0.01986871892054239) rectangle (-0.0000000001,0.19930340931194237);
\draw [samples=50,domain=-0.6:0,rotate around={0:(0,0)},xshift=0cm,yshift=0cm,line width=0.8pt,dash pattern=on 4pt off 3pt,color=qqwuqq] plot ({0.4472135954999579*(-1-(\x)^2)/(1-(\x)^2)},{0.4472135954999579*(-2)*(\x)/(1-(\x)^2)});
\draw [shift={(-0.45,0)},line width=1.5pt,color=ffqqqq]  plot[domain=0:3.141592653589793,variable=\t]({1*0.05*cos(\t r)+0*0.05*sin(\t r)},{0*0.05*cos(\t r)+1*0.05*sin(\t r)});
\draw [shift={(-0.44716337453437993,0)},line width=1.5pt,color=ffqqqq]  plot[domain=0:3.141592653589793,variable=\t]({1*0.002836625465620024*cos(\t r)+0*0.002836625465620024*sin(\t r)},{0*0.002836625465620024*cos(\t r)+1*0.002836625465620024*sin(\t r)});
\draw [shift={(-0.46286426423792154,0)},line width=1.5pt,color=ffqqqq]  plot[domain=0:3.141592653589793,variable=\t]({1*0.11934541092355616*cos(\t r)+0*0.11934541092355616*sin(\t r)},{0*0.11934541092355616*cos(\t r)+1*0.11934541092355616*sin(\t r)});
\draw [line width=0.7pt,dash pattern=on 4pt off 3pt] (-0.3435,0) -- (-0.3435,0.34430340931194237);
\draw [line width=0.7pt] (-0.15,0) -- (-0.15,0.34430340931194237);
\draw [line width=0.7pt,dash pattern=on 4pt off 3pt] (-0.44722,0) -- (-0.44722,0.34430340931194237);
\draw [line width=0.7pt,dash pattern=on 4pt off 3pt] (-0.5821,0) -- (-0.5821,0.34430340931194237);
\draw [line width=0.7pt] (-1,0) -- (0,0);
\begin{scriptsize}
\draw[color=qqwuqq] (0.9004810638190032,0.7555577912659127) node {$ec1$};
\end{scriptsize}

\begin{footnotesize}
\draw[color=black] (-0.12,0.185) node {$\beta=0$};
\draw[color=black] (-0.292,0.185) node {$\beta=-\frac{2y_0}{x_0+1}$};
\draw[color=black] (-0.632,0.185) node {$\beta=-\frac{2y_0}{x_0-1}$};
\draw[color=black] (-0.405,0.185) node {$\beta=-\frac{\sqrt{d}}{d}$};
\end{footnotesize}
\end{tikzpicture}
\caption{Possible walls for $\cI_0$ parametrized by solutions to Pell's equation}
\label{fig:treem}
\end{figure}
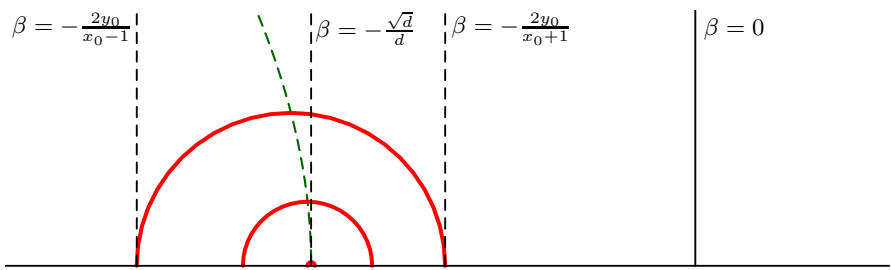

Hence the largest possible wall is associated to the minimal solution $(x_0,y_0)$, and the inequality $\epsilon_1(l)\leq\frac{2y_0}{x_0-1}$ follows.
\end{proof}

\begin{rem}
Upper bounds for $\epsilon_1(l)$ have been given by Ito for general abelian surfaces over $\bC$, using completely different techniques (see \cite[Proposition 4.4]{Ito2}). When $d$ is a perfect square, he already obtained the equality $\epsilon_1(l)=\frac{\sqrt{d}}{d}$ (and thus the expression for $h^0_{\cI_0,l}$).

On the other hand, for $d$ not a perfect square our upper bound refines the one given by Ito. Indeed, both bounds coincide for several values of $d$, but in general the inequality $\epsilon_1(l)\leq\frac{2y_0}{x_0-1}$ is stronger (e.g. $d=7,11,13,19,21,22,23,...$).
\end{rem}

One of the advantages of our approach is that it also controls the differentiability of the functions, which is meaningful in terms of M-regularity. Indeed, as a consequence of \autoref{main} and \autoref{regularity} we obtain:

\begin{cor}\label{correg}
Let $(S,l)$ be a polarized abelian surface satisfying the hypothesis of \autoref{main}.
\begin{enumerate}[\rm (1)]
    \item If $d$ is a perfect square, then $\cI_0\langle\frac{\sqrt{d}}{d}l\rangle$ is a GV-sheaf which is not M-regular.
    \item If $d$ is not a perfect square, then $\cI_0\langle\frac{2y_0}{x_0-1}l\rangle$ is M-regular.
\end{enumerate}
In particular, for $m\in\bZ_{>0}$ $\cI_0\langle\frac{1}{m}l\rangle$ is M-regular if and only if $m<\sqrt{d}$ (i.e. $m^2< d$).
\end{cor}
\begin{proof}
Only the last assertion is not directly deduced from \autoref{main} and \autoref{regularity}, as it also requires the following property: if $d$ is not a perfect square and $(x_0,y_0)$ is the minimal positive solution to $x^2-4d\cdot y^2=1$, then $\frac{2y_0}{x_0-1}\leq\frac{1}{m}$ for every integer $m<\sqrt{d}$. For the sake of clarity, we outline a proof of this inequality.

It suffices to check the case $m=\lfloor\sqrt{d}\rfloor$. To this end, we write $d=m^2+k$ for some $k\in\{1,\ldots,2m\}$ (which is possible since $m^2<d<(m+1)^2$). The inequality $\frac{2y_0}{x_0-1}\leq\frac{1}{m}$ is equivalent to $x_0^2\geq (2my_0+1)^2$, and hence to $ky_0\geq m$ (after using $x_0^2=4d\cdot y_0^2+1$ and $d=m^2+k$).

Since $y_0$ is a positive integer, the fulfillment of $ky_0\geq m$ is clear for $k\in\{m,\ldots,2m\}$. Hence we may assume $k\in\{1,\ldots,m-1\}$. For $k$ in this range, one observes that $y_0=1$ cannot happen; indeed, $y_0=1$ would imply that $4d\cdot y_0^2+1=4d+1=(2m)^2+4k+1$ is a perfect square, in contradiction to the inequalities
\[
(2m)^2<(2m)^2+4k+1\leq(2m)^2+4m-3<(2m+1)^2
\]

Therefore, $y_0\geq2$ for all $k\in\{1,\ldots,m-1\}$. In particular, the inequality $ky_0\geq m$ holds for all $k\in\{\lfloor\frac{m+1}{2}\rfloor,\ldots,m-1\}$, so we may assume $k\in\{1,\ldots,\lfloor\frac{m-1}{2}\rfloor\}$. 

For $k$ in this range, $y_0=2$ cannot happen; otherwise, $4d\cdot y_0^2+1=(4m)^2+16k+1$ would be a perfect square, contradicting
\[
(4m)^2<(4m)^2+16k+1\leq (4m)^2+8m-7<(4m+1)^2
\]

It follows that $y_0\geq 3$ for all $k\in\{1,\ldots,\lfloor\frac{m-1}{2}\rfloor\}$, which in particular proves $ky_0\geq m$ for all $k\in\{\lfloor\frac{m+2}{3}\rfloor,\ldots,\lfloor\frac{m-1}{2}\rfloor\}$. Repeating this process (proving $y_0\geq 4$ for all $k\in\{1,\ldots,\lfloor\frac{m-1}{3}\rfloor\}$, and so on) one obtains the desired inequality for all possible values of $k$.
\end{proof}

We point out that \autoref{correg} gives an affirmative answer, in the case of abelian surfaces, to a question posed by Ito (\cite[Remark 6.4]{Ito3}). By means of it, we prove \autoref{syzygies}:

\begin{proof}[Proof of \autoref{syzygies}]
Under the assumptions of \autoref{main}, $\cI_0\langle\frac{1}{2}l\rangle$ is $IT(0)$ for every $d\geq7$, as an immediate application of \autoref{regularity}.\eqref{IT0} and the upper bounds for $\epsilon_1(l)$. Thus the first assertion follows from \autoref{ito}.\eqref{ref:projnormal}.

If $d>(p+2)^2$ for some $p\geq1$, then $\cI_0\langle\frac{1}{p+2}l\rangle$ is M-regular by the last assertion of \autoref{correg}. Hence \autoref{ito}.\eqref{ref:Np} guarantees the property $(N_p)$ for representatives of $l$.
\end{proof}

\section{Lower bounds for \texorpdfstring{$\epsilon_1(l)$}{epsilon1(l)}}\label{sec:Lowerbound}

Let $d$ be a positive integer which is not a perfect square, and let $(x_0,y_0)$ be the minimal positive solution to $x^2-4d\cdot y^2=1$. In the sequel, we will assume that $\Char(\bK)$ divides neither $x_0^2$ nor $x_0^2-1$ (in particular, $\Char(\bK)\neq2$).

This section is devoted to prove \autoref{main}.\eqref{main3}, which in particular gives lower bounds for $\epsilon_1(l)$. Our approach is based on the following result (valid without the hypothesis of \autoref{main}): 

\begin{prop}\label{symcurve}
If $(S,l)$ is a $(1,d)$-polarized abelian surface and $L$ is a symmetric representative of $l$, then $h^0(S,\mu_{x_0}^*\cI_0\otimes L^{2x_0y_0})\geq x_0^2$. In other words, the linear system of curves $|L^{2x_0y_0}|$ has at least $x_0^2$ independent elements that contain all the $x_0$-torsion points of $S$.
\end{prop}

\begin{proof}
Since the subgroup $T\cong\left(\sfrac{\bZ}{x_0}\right)^4$ of $x_0$-torsion points is contained in
\[
K(L^{2x_0y_0})\cong\left(\sfrac{\bZ}{2x_0y_0}\oplus\sfrac{\bZ}{2dx_0y_0}\right)\times\widehat{\left(\sfrac{\bZ}{2x_0y_0}\oplus\sfrac{\bZ}{2dx_0y_0}\right)},
\]
we will use the representation of the theta group $\cG(L^{2x_0y_0})$ on $H^0(S,L^{2x_0y_0})$ to understand how translation by points of $T$ acts on the linear system $|L^{2x_0y_0}|$.

We consider the isomorphism of \autoref{mumford}.\eqref{repthetagroup}, which in particular identifies $\cG(L^{2x_0y_0})$ with $\bK^*\times K(L^{2x_0y_0})$ (with a noncommutative group operation), and $H^0(S,L^{2x_0y_0})$ with
\[
V(2x_0y_0,2dx_0y_0)=\{\bK\text{-valued functions on } \sfrac{\bZ}{2x_0y_0}\oplus\sfrac{\bZ}{2dx_0y_0}\}.
\]

Denote by $\{\delta_{j,k}\mid (j,k)\in\sfrac{\bZ}{2x_0y_0}\oplus\sfrac{\bZ}{2dx_0y_0}\}$ the canonical basis of $V(2x_0y_0,2dx_0y_0)$, that is: $\delta_{j,k}(l,m)=1$ if $(j,k)=(l,m)$, and $\delta_{j,k}(l,m)=0$ otherwise.

Moreover, let $\{a_1,a_2,a_3,a_4\}$ be the following basis of $T$ inside $K(L^{2x_0y_0})$:
\begin{itemize}[\textbullet]
    \item $a_1=(2y_0,0)$, $a_2=(0,2dy_0)$ in $\sfrac{\bZ}{2x_0y_0}\oplus\sfrac{\bZ}{2dx_0y_0}$.
    \item $a_3,a_4\in\Hom_\bZ(\sfrac{\bZ}{2x_0y_0}\oplus\sfrac{\bZ}{2dx_0y_0},\bK^*)$ are the homomorphisms given by
    \[
    a_3(1,0)=\xi,\;a_3(0,1)=1,\;a_4(1,0)=1,\;a_4(0,1)=\xi,
    \]
    where $\xi$ is a primitive $x_0$-th root of 1.
\end{itemize}

Consider the lifts $(1,a_i)\in\cG(L^{2x_0y_0})$ of $a_i$ (for $i=1,2,3,4$) to the theta group. According to the representation described in \autoref{mumford}.\eqref{repthetagroup}, they induce the endomorphisms
\[
\widetilde{a}_1:\delta_{j,k}\mapsto\delta_{j-2y_0,k}\;,\;\;\;\;\;\widetilde{a}_2:\delta_{j,k}\mapsto\delta_{j,k-2dy_0}\;,\;\;\;\;\;\widetilde{a}_3:\delta_{j,k}\mapsto\xi^j\delta_{j,k}\;,\;\;\;\;\;\widetilde{a}_4:\delta_{j,k}\mapsto\xi^k\delta_{j,k}
\]
on $H^0(S,L^{2x_0y_0})$. Recall that the projectivization of $\widetilde{a}_i$ on the linear system $|L^{2x_0y_0}|$ corresponds to (the dual of) the projectivity $t_{a_i}:\bP(H^0(S,L)^\vee)\to\bP(H^0(S,L)^\vee)$ extending $t_{a_i}:S\to S$.

Observe that $\widetilde{a}_3,\widetilde{a}_4$ are diagonalizable endomorphisms that commute (as corresponds to $a_3,a_4$ generating a totally isotropic subgroup of $K(L^{2x_0y_0})$). This implies that every eigenspace of $\widetilde{a}_3$ is an invariant subspace for $\widetilde{a}_4$, and conversely.

Therefore, we can find a decomposition
\[
H^0(S,L^{2x_0y_0})=\bigoplus_{\substack{l,m\in\{0,...,x_0-1\}}}  E_{(l,m)},
\]
where $E_{(l,m)}$ is a subspace of eigenvectors for both $\widetilde{a}_3$ and $\widetilde{a}_4$ (of eigenvalue $\xi^l$ for $\widetilde{a}_3$, and eigenvalue $\xi^m$ for $\widetilde{a}_4$). Explicitly, we have
\[
E_{(l,m)}=\langle\delta_{j,k}\mid j\equiv l\text{ and }k\equiv m\text{ (mod $x_0)$}\rangle,
\]
so every subspace $E_{(l,m)}$ has dimension $2y_0\cdot2dy_0=4dy_0^2=x_0^2-1$. 

The projectivization of $E_{(l,m)}$ represents a $(x_0^2-2)$-dimensional linear system $\mathcal{L}_{l,m}\subset|L^{2x_0y_0}|$, formed by curves which remain invariant under translation by points of the subgroup $\langle a_3,a_4\rangle\subset T$. In particular, any curve of $\mathcal{L}_{l,m}$ containing $\langle a_1,a_2\rangle\subset T$ automatically contains all of $T$.

Moreover, since $\gcd(x_0,2dy_0)=1$, it follows from the above description of $\widetilde{a}_1,\widetilde{a}_2$ that the subgroup $\langle(1,a_1),(1,a_2)\rangle\cong\left(\sfrac{\bZ}{x_0}\right)^2\subset\cG(L^{2x_0y_0})$ acts transitively on the set $\{E_{(l,m)}\}$.
Thus for our purposes it suffices to find a curve $C\in\mathcal{L}_{0,0}$ containing the $x_0^2$ points of $\langle a_1,a_2\rangle\subset T$. Indeed, the set of $x_0^2$ curves will be formed by one curve in each $\mathcal{L}_{l,m}$, obtained from $C$ by translation with the corresponding point of $\langle a_1,a_2\rangle$.

Since $L^{2x_0y_0}$ is totally symmetric, we may consider the involution of $H^0(S,L^{2x_0y_0})$
\[
\widetilde{i}:\delta_{j,k}\mapsto\delta_{-j,-k},
\]
whose projectivization extends the inversion $i:S\to S$ to a projectivity of $\bP(H^0(S,L^{2x_0y_0})^\vee)$.

The subspace $E_{(0,0)}$ is clearly invariant by this endomorphism, and the restriction $\widetilde{i}_{|E_{(0,0)}}$ satisfies:

\begin{itemize}[\textbullet]
    \item The subspace $E_{(0,0)}^1\subset E_{(0,0)}$ of eigenvectors of eigenvalue $1$ has dimension $2dy_0^2+2=\frac{x_0^2-1}{2}+2$. Explicitly, a basis of $E_{(0,0)}^1$ is given by
    \[
    \delta_{sx_0,tx_0}+\delta_{(2y_0-s)x_0,(2dy_0-t)x_0}
    \]
    for $s\in\{0,...,y_0\}$, and $t\in\{0,...,2dy_0-1\}$ (if $s\neq0,y_0$) or $t\in\{0,...,dy_0\}$ (if $s=0,y_0$).
    
    \vspace{1.5mm}
    
    \item The eigenspace $E_{(0,0)}^{-1}\subset E_{(0,0)}$ of eigenvalue $-1$ has dimension $2dy_0^2-2$, with basis
    \[
    \delta_{sx_0,tx_0}-\delta_{(2y_0-s)x_0,(2dy_0-t)x_0}
    \]
    for $s\in\{0,...,y_0\}$, and $t\in\{0,...,2dy_0-1\}$ (if $s\neq0,y_0$) or $t\in\{1,...,dy_0-1\}$ (if $s=0,y_0$).
\end{itemize}

The projectivization of $E_{(0,0)}^1$ defines a $(\frac{x_0^2-1}{2}+1)$-dimensional linear system $\mathcal{L}_{0,0}^1\subset\mathcal{L}_{0,0}$, formed by symmetric curves that remain invariant under translation by points of $\langle a_3,a_4\rangle\subset T$. 

Since $x_0$ is odd, the only 2-torsion point of $\langle a_1,a_2\rangle\cong\left(\sfrac{\bZ}{x_0}\right)^2$ is the origin of $S$; accordingly, points of $\langle a_1,a_2\rangle$ impose at most $\frac{x_0^2-1}{2}+1$ independent conditions on $\mathcal{L}_{0,0}^1$. It is thus possible to find a curve of $\mathcal{L}_{0,0}^1$ containing all the points of $\langle a_1,a_2\rangle\subset T$, which finishes the proof.
\end{proof}

\vspace{1.5mm}

\begin{proof}[Proof of \autoref{main}.\eqref{main3}]
\autoref{symcurve} shows (via Serre duality and cohomology and base change) that the sheaf $R^{2}\Phi_{\cP^{\vee}}((\mu_{x_0}^*\mathcal{I}_0\otimes L^{2x_0y_0})^{\vee})$ is nonzero. In virtue of the explicit expression for $h^0_{\cI_0,l}$ given in \eqref{explicitfunction}, this implies that $h^0_{\cI_0,l}(x)$ is positive for $x>\frac{2y_0}{x_0}$.

On the other hand, since $x_1=x_0^2+4dy_0^2$ and $y_1=2x_0y_0$, the equality $\frac{2y_0}{x_0}=\frac{2y_1}{x_1+1}$ holds.

Therefore, by \autoref{main} we conclude that only two expressions for $h^0_{\cI_0,l}$ are possible (those corresponding to the solutions $(x_0,y_0)$ and $(x_1,y_1)$). In particular, $\epsilon_1(l)\in\{\frac{2y_0}{x_0-1},\frac{2y_1}{x_1-1}\}$.
\end{proof}

\newpage

\begin{rem}\label{finalrem}\hfill
\begin{enumerate}[\rm (1)]
    \item\label{rationality} It follows, at least when $\Char(\bK)=0$, that $\epsilon_1(l)$ is rational under the assumptions of \autoref{main}. It would be interesting to know whether this holds true for every polarized abelian surface (or more generally, for every polarized abelian variety).

    \item\label{minimalsol} There are several examples of non-perfect squares $d$ where $\epsilon_1(l)$ is known for a general $(1,d)$-polarized (complex) abelian surface $(S,l)$ (see \cite[Example 5.11]{Ito2}); for all of them, there is an equality $\epsilon_1(l)=\frac{2y_0}{x_0-1}$.
Thus it seems reasonable to expect this for every non-perfect square $d$.
    
\noindent Assume the equality $\epsilon_1(l)=\frac{2y_1}{x_1-1}$ holds. According to the expression for $h^0_{\cI_0,l}$ given by \autoref{main}, for every $x>\frac{2y_1}{x_1+1}$ small enough we have
    \[
    h^0_{\cI_0,l}(x)=\frac{d(x_1+1)}{2}x^2-2dy_1\cdot x+\frac{x_1-1}{2}=dx_0^2\left(x-\frac{2y_0}{x_0}\right)^2,
    \]
    and then an elementary manipulation of \eqref{explicitfunction} shows that $R^{2}\Phi_{\cP^{\vee}}((\mu_{x_0}^*\mathcal{I}_0\otimes L^{2x_0y_0})^{\vee})$ is a 0-dimensional sheaf of length $x_0^2$.
    
    \noindent But note that \autoref{symcurve} precisely shows that, if  $R^{2}\Phi_{\cP^{\vee}}((\mu_{x_0}^*\mathcal{I}_0\otimes L^{2x_0y_0})^{\vee})$ is 0-dimensional, then it has length $\geq x_0^2$. Hence a slightly stronger version of \autoref{symcurve} (with $x_0^2+1$ independent curves on $|L^{2x_0y_0}|$, or with a curve in a translated linear system $|L^{2x_0y_0}\otimes\alpha|$ containing also $T$) would yield a contradiction, leading to a proof of $\epsilon_1(l)=\frac{2y_0}{x_0-1}$.
\end{enumerate}

\end{rem}

\vspace{1.5mm}

\bibliography{refer}
\bibliographystyle{alphaspecial}

\end{document}